\documentclass[11pt,reqno,a4paper]{amsart}

\usepackage{amsfonts, amsmath, amsthm, amssymb}
\usepackage[english]{babel}
\usepackage[mathcal]{eucal}

\allowdisplaybreaks

\title[]{On the Hausdorff measure of non-compactness for the parametrized Prokhorov metric}
\author{Ben Berckmoes}
\thanks{Ben Berckmoes is post doctoral fellow at the Fund for Scientific Research of Flanders (FWO)}

\date{}

\DeclareMathOperator*{\myinf}{in\vphantom{p}f}

\begin{document}

\maketitle

\newtheorem{pro}{Proposition}[section]
\newtheorem{lem}[pro]{Lemma}
\newtheorem{thm}[pro]{Theorem}
\newtheorem{de}[pro]{Definition}
\newtheorem{co}[pro]{Comment}
\newtheorem{no}[pro]{Notation}
\newtheorem{vb}[pro]{Example}
\newtheorem{vbn}[pro]{Examples}
\newtheorem{gev}[pro]{Corollary}
\newtheorem{vrg}[pro]{Question}
\newtheorem{rem}[pro]{Remark}
\newtheorem{lemA}{Lemma}

\begin{abstract}
We quantify Prokhorov's Theorem by establishing an explicit formula for the Hausdorff measure of non-compactness (HMNC) for the parametrized Prokhorov metric on the set of Borel probability measures on a Polish space. Furthermore, we quantify the Arzel{\`a}-Ascoli Theorem by obtaining an upper and a lower estimate for the HMNC for the uniform norm on the space of continuous maps of a compact interval into Euclidean N-space, using Jung's Theorem on the Chebyshev radius. Finally, we combine the obtained results to quantify the stochastic Arzel{\`a}-Ascoli Theorem by providing an upper and a lower estimate for the HMNC for the parametrized Prokhorov metric on the set of multivariate continuous stochastic processes.
\end{abstract}

\section{Introduction and statement of the main results}

For the basic probabilistic concepts and results, we refer the reader to any standard work on probability theory, such as e.g. \cite{Kal}.

Let $S$ be a Polish space, i.e. a separable completely metrizable topological space, and $\mathcal{P}(S)$ the collection of Borel probability measures on $S$, equipped with the weak topology $\tau_w$, i.e. the weakest topology for which each map 
\begin{equation*}
\mathcal{P}(S) \rightarrow \mathbb{R} : P \mapsto \int f dP, 
\end{equation*}
with $f : S \rightarrow \mathbb{R}$ bounded and continuous, is continuous. The space $\mathcal{P}(S)$ is known to be Polish.

We call a collection $\Gamma \subset \mathcal{P}(S)$ uniformly tight iff for each $\epsilon > 0$ there exists a compact set $K \subset S$ such that $P(S \setminus K) < \epsilon$ for all $P \in \Gamma$. 

The following celebrated result interrelates $\tau_w$-relative compactness with uniform tightness.

\begin{thm}[Prokhorov]\label{thm:cProkh}
A collection $\Gamma \subset \mathcal{P}(S)$ is $\tau_w$-relatively compact if and only if it is uniformly tight.
\end{thm}

Fix $N \in \mathbb{N}_0$ and let $\mathcal{C}$ be the space of continuous maps $x$ of the compact interval $[0,1]$ into Euclidean $N$-space $\mathbb{R}^N$, equipped with the uniform topology $\tau_\infty$, i.e. the topology derived from the uniform norm 
\begin{equation*}
\|x\|_\infty = \sup_{t \in [0,1]} \left|x(t)\right|, 
\end{equation*}
where $\left|\cdot\right|$ stands for the Euclidean norm. The space $\mathcal{C}$ is also known to be Polish.

Recall that a set $\mathcal{X} \subset \mathcal{C}$ is said to be uniformly bounded iff there exists a constant $M > 0$ such that $\left|x(t)\right| \leq M$ for all $x \in \mathcal{X}$ and $t \in \left[0,1\right]$, and uniformly equicontinuous iff for each $\epsilon > 0$ there exists $\delta > 0$ such that $\left|x(s) - x(t)\right| < \epsilon$ for all $x \in \mathcal{X}$ and $s,t \in \left[0,1\right]$ with $\left|s - t\right| < \delta$.

In this setting, the following theorem is a classic (\cite{L93}).

\begin{thm}[Arzel{\`a}-Ascoli]\label{thm:cAA}
A collection $\mathcal{X} \subset \mathcal{C}$ is $\tau_\infty$-relatively compact if and only if it is uniformly bounded and uniformly equicontinuous.
\end{thm}

Let $\Omega = (\Omega,\mathcal{F},\mathbb{P})$ be a fixed probability space. Throughout, a continuous stochastic process (csp) is a Borel measurable map of $\Omega$ into $\mathcal{C}$, and we consider on the set of csp's the weak topology $\tau_w$, i.e. the topology with open sets $\{\xi \text{ csp} \mid \mathbb{P}_\xi \in \mathcal{G}\}$, where $\mathbb{P}_\xi$ is the probability distribution of $\xi$ and $\mathcal{G}$ is a $\tau_w$-open set in $\mathcal{P}\left(\mathcal{C}\right)$. 

A collection $\Xi$ of csp's is said to be stochastically uniformly bounded iff for each $\epsilon > 0$ there exists $M > 0$ such that $\mathbb{P}\left(\|\xi\|_\infty > M\right) < \epsilon$ for all $\xi \in \Xi$, and stochastically uniformly equicontinuous iff for all $\epsilon, \epsilon^\prime > 0$ there exists $\delta > 0$ such that $\mathbb{P}\left(\sup_{\left|s - t\right| < \delta} \left|\xi(s) - \xi(t)\right| \geq \epsilon\right) < \epsilon^\prime$ for all $\xi \in \Xi$, the supremum taken over all $s,t \in [0,1]$ for which $\left|s - t\right| < \delta$.

It is not hard to see that combining Theorem \ref{thm:cProkh} and Theorem \ref{thm:cAA} yields the following stochastic version of Theorem \ref{thm:cAA}, which plays a crucial role in the development of functional central limit theory.

\begin{thm}[stochastic Arzel{\`a}-Ascoli]\label{thm:csAA}
A collection $\Xi$ of csp's is $\tau_w$-relatively compact if and only if it is stochastically uniformly bounded and stochastically uniformly equicontinuous.
\end{thm}

In a complete metric space $(X,d)$, the Hausdorff measure of non-compactness of a set $A \subset X$  (\cite{BG80},\cite{WW96}) is given by 
\begin{equation*}
\mu_{\text{\upshape{H}},d}(A) = \myinf_{F} \sup_{x \in A} \myinf_{y \in F} d(x,y), 
\end{equation*}
the first infimum running through all finite sets $F \subset X$. It is well known that $A$ is $d$-bounded if and only if $\mu_{\textrm{\upshape{H}},d}(A) < \infty$, and $d$-relatively compact if and only if $\mu_{\textrm{\upshape{H}},d}(A) = 0$.

Fix a complete metric $d$ metrizing the topology of the Polish space $S$. The Prokhorov distance with parameter $\lambda \in \mathbb{R}^+_0$ between probability measures $P,Q \in \mathcal{P}(S)$ (\cite{R91}) is defined as the infimum of all numbers $\alpha \in \mathbb{R}^+_0$ for which the inequality 
\begin{equation*}
P\left(A\right) \leq Q\left(A^{(\lambda \alpha)}\right) + \alpha
\end{equation*}
holds for all Borel sets $A \subset S$, where 
\begin{equation*}
A^{(\epsilon)} = \left\{x \in S \mid \inf_{a \in A} d(a,x) \leq \epsilon\right\}. 
\end{equation*}
This distance is denoted by $\rho_\lambda(P,Q)$. It defines a complete metric on $\mathcal{P}(S)$ which induces the weak topology $\tau_w$. It is also known that $\rho_{\lambda_1} \leq \rho_{\lambda_2}$ if $\lambda_1 \geq \lambda_2$, and that 
\begin{equation*}
\sup_{\lambda \in \mathbb{R}^+_0} \rho_\lambda(P,Q) = \sup_A \left|P(A) - Q(A)\right|,
\end{equation*}
the supremum being taken over all Borel sets $A \subset S$.

For a collection $\Gamma \subset \mathcal{P}(S)$, we define the measure of non-uniform tightness as 
\begin{equation*}
\mu_{\text{\upshape ut}}(\Gamma) = \sup_{\epsilon > 0} \myinf_{Y} \sup_{P \in \Gamma} P\left(S \setminus \cup_{y \in Y} B(y,\epsilon)\right), 
\end{equation*}
where the infimum runs through all finite sets $Y \subset S$, and 
\begin{equation*}
B(y,\epsilon) = \left\{x \in S \mid d(y,x) < \epsilon\right\}.
\end{equation*}
It is clear that $\mu_{\text{\upshape ut}}(\Gamma) = 0$ if $\Gamma$ is uniformly tight. But the converse holds as well. Indeed, suppose that $\mu_{\text{\upshape ut}}(\Gamma) = 0$ and fix $\epsilon > 0$. Then, for each $n \in \mathbb{N}_0$, choose a finite set $Y_n \subset S$ such that 
\begin{equation*}
P\left(S \setminus \cup_{y \in Y_n} B(y,1/n)\right) < \epsilon/2^n 
\end{equation*}
for all $P \in \Gamma$. Put 
\begin{equation*}
K = \cap_{n \in \mathbb{N}_0} \cup_{y \in Y_n} B^\star(y,1/n), 
\end{equation*}
with $B^\star(y,1/n)$ the closure of $B(y,1/n)$. Then $K$ is a compact set such that $P(S \setminus K) < \epsilon$ for all $P \in \Gamma$. We conclude that $\Gamma$ is uniformly tight. The measure $\mu_{\text{\upshape{ut}}}$ is slightly weaker than the weak measure of tightness studied in \cite{BLV11}.

By the previous considerations, we know that a set $\Gamma \subset \mathcal{P}(S)$ is $\tau_w$-relatively compact if and only if $\mu_{\text{\upshape{H}},\rho_\lambda}(\Gamma) = 0$ for each $\lambda \in \mathbb{R}^+_0$, and uniformly tight if and only if $\mu_{\text{\upshape ut}}(\Gamma) = 0$. Therefore, Theorem \ref{thm:qProkh}, our first main result, which provides a quantitative relation between the numbers $\mu_{\text{\upshape{H}},\rho_\lambda}(\Gamma)$ and $\mu_{\text{\upshape ut}}(\Gamma)$, is a strict generalization of Theorem \ref{thm:cProkh}. The proof is given in Section \ref{sec:qProkh}.

\begin{thm}[quantitative Prokhorov]\label{thm:qProkh}
For a collection $\Gamma \subset \mathcal{P}(S)$, $$\sup_{\lambda \in \mathbb{R}^+_0}\mu_{\text{\upshape{H}},\rho_\lambda}(\Gamma) = \mu_{\text{\upshape ut}}(\Gamma).$$ 
\end{thm}

From now on, we consider on the space $\mathcal{C}$ the uniform metric, derived from the uniform norm, and for a set $\mathcal{X} \subset \mathcal{C}$, we let $\mu_{\text{\upshape H},\infty}(\mathcal{X})$ stand for the Hausdorff measure of non-compactness, more precisely, 
\begin{equation*}
\mu_{\text{\upshape H},\infty}(\mathcal{X}) = \myinf_{\mathcal{F}} \sup_{x \in \mathcal{X}} \myinf_{y \in \mathcal{F}} \|x - y\|_\infty, 
\end{equation*}
the infimum taken over all finite sets $\mathcal{F} \subset \mathcal{C}$. Clearly, $\mathcal{X}$ is $\tau_\infty$-relatively compact if and only if $\mu_{\text{\upshape H},\infty}(\mathcal{X}) = 0$.

The measure of non-uniform equicontinuity of $\mathcal{X} \subset \mathcal{C}$ is defined by 
\begin{equation*}
\mu_{\text{\upshape{uec}}}(\mathcal{X}) = \myinf_{\delta > 0} \sup_{x \in \mathcal{X}} \sup_{\left|s - t\right|< \delta} \left|x(s) - x(t)\right|, 
\end{equation*}
the second supremum running through all $s,t \in [0,1]$ with $\left|s - t\right| < \delta$. One readily sees that $\mathcal{X}$ is uniformly equicontinuous if and only if $\mu_{\textrm{\upshape uec}}(\mathcal{X}) = 0$. In \cite{BG80} it was shown that $\mu_{\textrm{\upshape{uec}}}$ is a measure of non-compactness on the space $\mathcal{C}$ (Theorem 11.2).

Theorem \ref{thm:qAA}, our second main result, entails that the measures $\mu_{\text{\upshape H},\infty}$ and $\mu_{\textrm{\upshape uec}}$ are Lipschitz equivalent on the collection of uniformly bounded subsets of $\mathcal{C}$, and thus it strictly generalizes Theorem \ref{thm:cAA}. The proof, which hinges upon a classical result of Jung's on the Chebyshev radius, is given in Section \ref{sec:qAA}.

\begin{thm}[quantitative Arzel{\`a}-Ascoli]\label{thm:qAA}
For $\mathcal{X} \subset \mathcal{C}$,
$$\frac{1}{2} \mu_{\textrm{\upshape{uec}}}(\mathcal{X}) \leq \mu_{\textrm{\upshape{H}},\infty}(\mathcal{X}).$$
Suppose, in addition, that $\mathcal{X}$ is uniformly bounded. Then
$$\mu_{\textrm{\upshape{H}},\infty}(\mathcal{X}) \leq  \left(\frac{N}{2N + 2}\right)^{1/2} \mu_{\textrm{\upshape{uec}}}(\mathcal{X}).$$ 
In particular, if $N = 1$, then
$$\mu_{\textrm{\upshape{H}},\infty}(\mathcal{X}) = \frac{1}{2} \mu_{\textrm{\upshape{uec}}}(\mathcal{X}),$$
and, regardless of $N$,
$$\mu_{\textrm{\upshape{H}},\infty}(\mathcal{X}) \leq  \frac{\sqrt{2}}{2} \mu_{\textrm{\upshape{uec}}}(\mathcal{X}).$$
\end{thm}

We transport the parametrized Prokhorov metric from $\mathcal{P}(\mathcal{C})$ to the collection of csp's via their probability distributions. Thus, for csp's $\xi$ and $\eta$,
\begin{equation*}
\rho_\lambda(\xi,\eta) = \rho_\lambda\left(\mathbb{P}_\xi,\mathbb{P}_\eta\right).
\end{equation*}
Note that a set of csp's $\Xi$ is $\tau_\omega$-relatively compact if and only if $\mu_{\text{\upshape{H}},\rho_\lambda}(\Xi) = 0$ for all $\lambda \in \mathbb{R}^+_0$.

For a set of csp's $\Xi$, the measure of non-stochastic uniform boundedness is given by
\begin{equation*}
\mu_{\text{\upshape{sub}}}(\Xi) = \myinf_{M \in \mathbb{R}^+_0} \sup_{\xi \in \Xi} \mathbb{P}(\|\xi\|_\infty > M),
\end{equation*}
and the measure of non-stochastic uniform equicontinuity by
\begin{equation*}
\mu_{\text{\upshape{suec}}}(\Xi) = \sup_{\epsilon > 0} \myinf_{\delta> 0} \sup_{\xi \in \Xi} \mathbb{P}\left(\sup_{\left|s - t\right| < \delta} \left|\xi(s) - \xi(t)\right| \geq \epsilon\right),
\end{equation*}
where the third supremum is taken over all $s,t \in [0,1]$ with $\left|s - t\right|< \delta$. It is easily seen that $\Xi$ is stochastically uniformly bounded if and only if $\mu_{\text{\upshape{sub}}}(\Xi) = 0$, and  stochastically uniformly equicontinuous if and only if $\mu_{\text{\upshape{suec}}}(\Xi) = 0$. The measure $\mu_{\text{\upshape{suec}}}$ was studied in \cite{BLV11}.

In Section \ref{sec:qsAA}, we prove that combining Theorem \ref{thm:qProkh} and Theorem \ref{thm:qAA} leads to Theorem \ref{thm:qsAA}, our third main result, which gives an upper and a lower bound for $\sup_{\lambda \in \mathbb{R}^+_0} \mu_{\text{\upshape H},\rho_\lambda}$ in terms of $\mu_{\text{\upshape sub}}$ and $\mu_{\text{\upshape suec}}$. Theorem \ref{thm:qsAA} strictly generalizes Theorem \ref{thm:csAA}.

\begin{thm}[quantitative stochastic Arzel{\`a}-Ascoli]\label{thm:qsAA}
Let $\Xi$ be a collection of csp's. Then
$$\max\{\mu_{\text{\upshape{sub}}}(\Xi), \mu_{\text{\upshape{suec}}}(\Xi)\} \leq \sup_{\lambda \in \mathbb{R}^+_0}\mu_{\textrm{\upshape{H}},\rho_\lambda}(\Xi) \leq \mu_{\text{\upshape{sub}}}(\Xi) + \mu_{\text{\upshape{suec}}}(\Xi).$$
In particular, if $\Xi$ is stochastically uniformly bounded, then
$$\sup_{\lambda \in \mathbb{R}^+_0}\mu_{\textrm{\upshape{H}},\rho_\lambda}(\Xi) = \mu_{\text{\upshape{suec}}}(\Xi),$$
and, if $\Xi$ is stochastically uniformly equicontinuous, then
$$\sup_{\lambda \in \mathbb{R}^+_0}\mu_{\textrm{\upshape{H}},\rho_\lambda}(\Xi) = \mu_{\text{\upshape{sub}}}(\Xi).$$
\end{thm}

\section{Proof of Theorem \ref{thm:qProkh}}\label{sec:qProkh}

For a collection $\Gamma \subset \mathcal{P}(S)$, put
$$p_\Gamma = \sup_{\lambda \in \mathbb{R}^+_0}\mu_{\text{\upshape{H}},\rho_\lambda}(\Gamma)$$
and 
$$t_\Gamma = \mu_{\text{\upshape ut}}(\Gamma).$$

We first show that 
$$p_\Gamma \leq t_\Gamma$$ 
with an argument which essentially refines the first part of the proof of Theorem 4.9 in \cite{BLV11}.

Fix $\lambda \in \mathbb{R}^+_0$, $\epsilon > 0$, and choose pairwise disjoint Borel sets 
$$A_1, \ldots, A_n \subset S,$$ 
with diameter less than $\lambda \epsilon$, such that
$$\forall P \in \Gamma : P\left(S \setminus \cup_{i=1}^n A_i\right) \leq t_\Gamma + \epsilon/2.$$
Then, for each $i \in \{1,\ldots,n\}$, pick $x_i \in A_i$, and, assuming without loss of generality that $S \setminus \cup_{i=1}^n A_i$ is non-empty, $x_{n+1} \in S \setminus \cup_{i=1}^n A_i$. Finally, fix $m \in \mathbb{N}_0$ such that 
$$n/m \leq \epsilon/2,$$
and let $\Phi$ stand for the finite collection of Borel probability measures on $S$ of the form 
$$Q = \sum_{i=1}^{n+1} (k_i/m) \delta_{x_i},$$
where the $k_i$ range in $\{0,\ldots,m\}$ such that 
$$\sum_{i=1}^{n + 1} k_i = m,$$
and $\delta_{x_i}$ stands for the Dirac probability measure putting all its mass on $x_i$.

We now claim that
$$\forall P \in \Gamma, \exists Q \in \Phi : \rho_\lambda(P,Q) \leq t_\Gamma + \epsilon,$$
which finishes the proof of the desired inequality.

To prove the claim, take $P \in \Gamma$, and construct
$$Q = \sum_{i=1}^{n + 1} (k_i /m) \delta_{x_i}$$
in $\Phi$ such that 
$$P(A_i) \leq k_i/m + 1/m$$
for all $i \in \{1,\ldots,n\}$. For a Borel set $A \subset S$, let $I$ stand for the set of those $i \in \{1,\ldots,n\}$ for which $A_i \cap A$ is non-empty. Then we derive from the calculation
\begin{eqnarray*}
P(A) &\leq& P\left(\cup_{i \in I} A_i\right) + P(S \setminus \cup_{i=1}^n A_i)\\
&\leq& \sum_{i \in I} P(A_i) + t_\Gamma + \epsilon/2\\
&\leq& \sum_{i \in I} \left(k_i/m + 1/m\right) + t_\Gamma + \epsilon/2\\
&\leq& Q\left(\cup_{i \in I} A_i\right) + n/m + t_\Gamma + \epsilon/2\\
&\leq& Q\left(A^{(\lambda(t_\Gamma + \epsilon))}\right) + t_\Gamma + \epsilon\\
\end{eqnarray*}
that 
$$\rho_\lambda(P,Q) \leq t_\Gamma + \epsilon,$$
establishing the claim.

We now show that 
$$t_\Gamma \leq p_\Gamma.$$

Fix $\epsilon, \epsilon^\prime > 0$. Choose $\lambda \in \mathbb{R}^+_0$  such that
$$\lambda \left(p_\Gamma + \epsilon^\prime/2\right) \leq \epsilon/2,$$
and take a finite collection $\Phi \subset \mathcal{P}(S)$ such that for each $P \in \Gamma$ there exists $Q \in \Phi$ for which
$$ \rho_{\lambda}(P,Q) \leq  \mu_{\text{\upshape H},\rho_{\lambda}}(\Gamma)+ \epsilon^\prime/2 \leq p_\Gamma + \epsilon^\prime/2.$$ 
The collection $\Phi$ being finite, we can pick a finite set $Y \subset S$ such that 
$$\forall Q \in \Phi : Q \left(S \setminus \cup_{y \in Y} B(y,\epsilon/2)\right) \leq \epsilon^\prime/2.$$

We claim that
$$\forall P \in \Gamma : P\left(S \setminus \cup_{y \in Y} B(y,\epsilon)\right) \leq p_\Gamma + \epsilon^\prime,$$
proving the desired inequality. 

To establish the claim, take $P \in \Gamma$, and let  $Q$ be a probability measure in $\Phi$ such that 
$$\rho_{\lambda}(P,Q) \leq p_\Gamma + \epsilon^\prime/2.$$
But then
\begin{eqnarray*} 
\lefteqn{P\left(S \setminus \cup_{y \in Y} B(y,\epsilon)\right)}\\
&\leq& Q\left(\left(S \setminus \cup_{y \in Y} B(y,\epsilon)\right)^{\left(\lambda\left(p_\Gamma+\epsilon^\prime/2\right)\right)}\right) + p_\Gamma + \epsilon^\prime/2\\
&\leq& Q\left(\left(S \setminus \cup_{y \in Y} B(y,\epsilon)\right)^{(\epsilon/2)}\right) + p_\Gamma + \epsilon^\prime/2\\
&\leq& Q\left(S \setminus \cup_{y \in Y} B(y,\epsilon/2)\right) + p_\Gamma + \epsilon^\prime/2\\
&\leq& p_\Gamma + \epsilon^\prime,
\end{eqnarray*}
which finishes the proof of the claim.

\section{Proof of Theorem \ref{thm:qAA}}\label{sec:qAA}

Before writing down the proof of Theorem \ref{thm:qAA}, we give the required preparation.

For a bounded set $A \subset \mathbb{R}^N$, the diameter is given by
$$\text{\upshape{diam}}(A) = \sup_{x,y \in A} \left|x-y\right|,$$
and the Chebyshev radius by
$$r(A) = \myinf_{x \in \mathbb{R}^N} \sup_{y \in A} \left|x - y\right|.$$
It is well known that for each bounded set $A \subset \mathbb{R}^N$ there exists a unique $x_A \in \mathbb{R}^N$ such that 
\begin{displaymath}
\sup_{y \in A} \left|x_A - y\right| = r(A).
\end{displaymath}
The point $x_A$ is called the Chebyshev center of $A$. A good exposition of the previous notions in a general normed vector space can be found in \cite{H72}, Section 33.

Theorem \ref{thm:Jung} provides a relation between the diameter and the Chebyshev radius of a bounded set in $\mathbb{R}^N$. A beautiful proof can be found in \cite{BW41}. For extensions of the result, we refer to \cite{A85}, \cite{AFS00}, \cite{R02}, and \cite{NN06}.

\begin{thm}[Jung]\label{thm:Jung}
Let $A \subset \mathbb{R}^N$ be a bounded set. Then
$$\frac{1}{2} \text{\upshape{diam}}(A) \leq r(A) \leq \left(\frac{N}{2N + 2}\right)^{1/2} \textrm{\upshape{diam}}(A).$$
\end{thm}

We need two more simple lemmas on linear interpolation.

For $c_0 \in \mathbb{R}^N$ and $r \in \mathbb{R}_0^+$, we denote the closed ball with center $c_0$ and radius $r$ by $B^\star(c_0,r)$.

\begin{lem}\label{lem:Tech1}
Consider $c_1,c_2 \in \mathbb{R}^N$ and $r \in \mathbb{R}^+_0$, and assume that 
$$B^\star(c_1,r) \cap B^\star(c_2,r) \neq \emptyset.$$ 
Let $L$ be the $\mathbb{R}^N$-valued map on the compact interval $\left[\alpha,\beta\right]$ defined by 
$$L(t) = \frac{\beta - t}{\beta - \alpha} c_1 + \frac{t - \alpha}{\beta - \alpha} c_2.$$
Then, for all $t \in \left[\alpha,\beta\right]$ and $y \in B^\star(c_1,r) \cap B^\star(c_2,r)$,
$$\left|L(t) - y\right| \leq r.$$
\end{lem}

\begin{proof}
The calculation
\begin{eqnarray*}
\left|L(t) - y\right| &=& \left|\frac{\beta - t}{\beta - \alpha} (c_1 - y) + \frac{t - \alpha}{\beta - \alpha} (c_2 - y)\right|\\
&\leq& \frac{\beta - t}{\beta - \alpha} \left|c_1 - y\right| + \frac{t - \alpha}{ \beta - \alpha} \left|c_2 - y\right|\\
&\leq& \frac{\beta - t}{\beta - \alpha} r + \frac{t - \alpha}{\beta - \alpha} r\\
&=& r
\end{eqnarray*}
proves the lemma.
\end{proof}

\begin{lem}\label{lem:Tech2}
Consider $c_1,c_2,y_1,y_2 \in \mathbb{R}^N$ and $\epsilon > 0$, and suppose that 
$$\left|c_1 - y_1\right| \leq \epsilon$$ 
and 
$$\left|c_2 - y_2\right| \leq \epsilon.$$
Let $L$ and $M$ be the $\mathbb{R}^N$-valued maps on the compact interval $\left[\alpha,\beta\right]$ defined by
\begin{displaymath}
L(t) = \frac{\beta - t}{\beta - \alpha} c_1 + \frac{t - \alpha}{\beta - \alpha} c_2
\end{displaymath}
and
\begin{displaymath}
M(t) = \frac{\beta - t}{\beta - \alpha} y_1 + \frac{t - \alpha}{\beta - \alpha} y_2.
\end{displaymath}
Then 
\begin{displaymath}
\|L - M\|_\infty \leq \epsilon.
\end{displaymath}
\end{lem}

\begin{proof}
This is analogous to the proof of Lemma \ref{lem:Tech1}.
\end{proof}

\begin{proof}[Proof of Theorem \ref{thm:qAA}]

We first prove that 
$$\frac{1}{2} \mu_{\textrm{\upshape{uec}}}\left(\mathcal{X}\right) \leq  \mu_{\text{\upshape{H}},\infty}\left(\mathcal{X}\right).$$ 
Let $\alpha > 0$ be so that $\mu_{\text{\upshape{H}},\infty}\left(\mathcal{X}\right) < \alpha$. Then there exists a finite set $\mathcal{F} \subset \mathcal{C}$ such that for all $x \in \mathcal{X}$ there exists $y \in \mathcal{F}$ for which $\|y - x\|_\infty \leq \alpha$. Take $\epsilon > 0$. Since $\mathcal{F}$ is uniformly equicontinuous, there exists $\delta > 0$ so that
\begin{equation} 
\forall y \in \mathcal{F}, \forall s, t \in \left[0,1\right] : \left|s - t\right| < \delta \Rightarrow \left|y(s) - y(t)\right| \leq \epsilon.\label{eq:FzEC}
\end{equation}
Now, for $x \in \mathcal{X}$, choose $y \in \mathcal{F}$ such that 
\begin{equation}
\|y - x\|_\infty \leq \alpha.\label{eq:2gCloseTof}
\end{equation}
Then, for $s , t \in \left[0,1\right]$ with $\left|s - t\right| < \delta$, we have, by (\ref{eq:FzEC}) and (\ref{eq:2gCloseTof}),
$$\left|x(s) - x(t)\right| \leq \left|x(s) - y(s)\right| + \left|y(s) - y(t)\right| + \left|y(t) - x(t)\right| \leq 2 \alpha + \epsilon,$$
which, by the arbitrariness of $\epsilon$, reveals that  $\mu_{\textrm{\upshape{uec}}}\left(\mathcal{X}\right) \leq 2\alpha$, and thus, by the arbitrariness of $\alpha$, the inequality 
$$\frac{1}{2} \mu_{\textrm{\upshape{uec}}}\left(\mathcal{X}\right) \leq \mu_{\text{\upshape{H}},\infty}\left(\mathcal{X}\right)$$
holds.

Next, assume that $\mathcal{X} \subset \mathcal{C}$ is uniformly bounded. We show that 
$$\mu_{\text{\upshape{H}},\infty}(\mathcal{X}) \leq \left(\frac{N}{2N + 2}\right)^{1/2} \mu_{\text{\upshape{uec}}}(\mathcal{X}).$$
Fix $\epsilon > 0$. Then, $\mathcal{X}$ being uniformly bounded, we can take a constant $M> 0$ such that 
\begin{equation}
\forall x \in \mathcal{X}, \forall t \in \left[0,1\right] : \left|x(t)\right| \leq M.\label{bdd}
\end{equation}
Pick a finite set $Y \subset \mathbb{R}^N$ for which
\begin{equation}
\forall z \in B^\star(0,3M), \exists y \in Y : \left|y - z\right| \leq \epsilon.\label{YDense}
\end{equation}
Now let 
\begin{equation}
0 < \alpha \leq 2M\label{eq:choiceAlpha}
\end{equation}
be so that $\mu_{\textrm{\upshape{uec}}}(\mathcal{X}) < \alpha$, i.e. there exists $\delta > 0$ for which 
\begin{equation}
\forall x \in \mathcal{X}, \forall s,t \in \left[0,1\right] : \left|s - t\right| <  \delta \Rightarrow \left|x(s) - x(t)\right| \leq \alpha.\label{eq:eqct}
\end{equation}
Then choose points
$$0 = t_0 < t_1 < \ldots < t_{2n} < t_{2n + 1} = 1,$$
put
\begin{eqnarray*}
I_0 &=& \left[0,t_2\right[,\\
I_k &=& \left]t_{2k - 1}, t_{2k + 2}\right[ \textrm{ if } k \in \left\{1,\ldots,n-1\right\},\\
I_n &=& \left]t_{2n - 1}, 1\right],
\end{eqnarray*}
and assume that we have made this choice such that
\begin{equation}
\forall k \in \left\{0, \ldots, n\right\} : \text{\upshape{diam}}(I_k) < \delta.\label{eq:diam}
\end{equation}
Furthermore, for each $(y_0, \ldots, y_{2n+ 1}) \in Y^{2n+2}$, let  $L_{(y_0, \ldots, y_{2n+1})}$ be the $\mathbb{R}^N$-valued map on $\left[0,1\right]$ defined by
\begin{eqnarray*}
L_{(y_0, \ldots, y_{2n+1})}(t) = \left\{\begin{array}{clrrrrrrrrrrrrrr}      
\frac{t_1 - t}{t_1 - t_0} y_0 + \frac{t - t_0}{t_1 - t_0} y_1 &\textrm{if}&  t \in \left[t_0,t_1\right]\\       
\vdots\\
\frac{t_{k+1} - t}{t_{k+1} - t_k} y_k + \frac{t - t_k}{t_{k+1} - t_k} y_{k + 1} & \textrm{if}& t \in \left[t_{k},t_{k+1}\right]\\
\vdots\\
\frac{t_{2n+1} - t}{t_{2n+1} - t_{2n}} y_{2n} + \frac{t - t_{2n}}{t_{2n+1} - t_{2n}} y_{2n + 1} &\textrm{if}& t \in \left[t_{2n},t_{2n + 1}\right]
\end{array}\right.,
\end{eqnarray*}
and put
\begin{displaymath}
\mathcal{F} = \left\{L_{(y_0,\ldots,y_{2n+1})} \mid (y_0, \ldots, y_{2n + 1}) \in Y^{2n + 2}\right\}.
\end{displaymath}
Then $\mathcal{F}$ is a finite subset of $\mathcal{C}$. Now fix $x \in \mathcal{X}$ and let $c_{x,k}$ stand for the Chebyshev center of $x(I_k)$ for each $k \in \left\{0,\ldots, n\right\}$. It follows from (\ref{eq:eqct}) and (\ref{eq:diam}) that $\textrm{\upshape{diam}}f(I_k) \leq \alpha$, and thus, by Theorem \ref{thm:Jung}, 
\begin{equation}
\forall k \in \left\{0,\ldots,n\right\} : \sup_{t \in I_k} \left|c_{x,k} - x(t)\right| \leq \left(\frac{N}{2N + 2}\right)^{1/2} \alpha.\label{eq:JApp}
\end{equation}
Let $\widetilde{x}$ be the $\mathbb{R}^N$-valued map on $\left[0,1\right]$ defined by
\begin{displaymath}
\widetilde{x}(t) = \left\{\begin{array}{clrrrrrrrrrrrrrr}      
c_{x,0}& \textrm{if}&  t \in \left[t_0,t_1\right]\\       
\frac{t_2 - t}{t_2 - t_1} c_{x,0}+ \frac{t - t_1}{t_2 - t_1} c_{x,1}& \textrm{if}& t \in \left[t_1,t_2\right]\\
c_{x,1} & \textrm{if}& t \in \left[t_2,t_3\right[\\
\frac{t_4 - t}{t_4 - t_3} c_{x,1}+ \frac{t - t_3}{t_4 - t_3} c_{x,2} &\textrm{if}& t \in \left[t_3,t_4\right]\\
\vdots\\
\frac{t_{2k} - t}{t_{2k} - t_{2k-1}} c_{x,k-1}+ \frac{t - t_{2k-1}}{t_{2k} - t_{2k-1}} c_{x,k} & \textrm{if}& t \in \left[t_{2k-1},t_{2k}\right]\\
c_{x,k} & \textrm{if}& t \in \left[t_{2k},t_{2k+1}\right]\\
\frac{t_{2 k + 2} - t}{t_{2 k + 2} - t_{2 k + 1}} c_{x,k}+ \frac{t - t_{2 k + 1}}{t_{2 k + 2} - t_{2 k + 1}} c_{x,k+1} &\textrm{if}& t \in \left[t_{2k+1},t_{2k+2}\right]\\
\vdots\\
\frac{t_{2 n - 2} - t}{t_{2 n - 2} - t_{2 n - 3}} c_{x,n - 2}+ \frac{t - t_{2n - 3}}{t_{2n - 2} - t_{2n - 3}} c_{x,n - 1} &\textrm{if}& t \in \left[t_{2n-3},t_{2n-2}\right]\\
c_{x,n-1} & \textrm{if}& t \in \left[t_{2n-2},t_{2n-1}\right]\\
\frac{t_{2 n} - t}{t_{2 n} - t_{2 n - 1}} c_{x,n-1}+ \frac{t - t_{2 n -1}}{t_{2 n} - t_{2 n - 1}} c_{x,n}  &\textrm{if}& t \in \left[t_{2n-1},t_{2n}\right]\\
c_{x,n} & \textrm{if}& t \in \left[t_{2n},t_{2n+1}\right]
\end{array}\right..
\end{displaymath}
Then (\ref{eq:JApp}) and Lemma \ref{lem:Tech1} learn that 
\begin{equation}
\|\widetilde{x} - x\|_{\infty} \leq \left(\frac{N}{2N + 2}\right)^{1/2} \alpha.\label{ftcf} 
\end{equation} 
Also, it easily follows from (\ref{bdd}), (\ref{eq:choiceAlpha}), and (\ref{ftcf}) that $\|\widetilde{x}\|_{\infty} \leq 3M$, and thus (\ref{YDense}) allows us to choose $\left(y_0, \ldots, y_{2n + 1}\right) \in Y^{2n + 2}$ such that 
\begin{equation}
\forall k \in \left\{0,\ldots, 2n + 1\right\} : \left|y_k - \widetilde{x}(t_k)\right| \leq \epsilon.\label{ykc}
\end{equation}
Combining (\ref{ykc}) and Lemma \ref{lem:Tech2} reveals that
\begin{equation}
\|L_{(y_0,\ldots,y_{2n + 1})} - \widetilde{x}\|_{\infty} \leq \epsilon.\label{ltft}
\end{equation}
But then we have found $L_{(y_0,\ldots,y_{2n + 1})}$ in $\mathcal{F}$ for which, by (\ref{ftcf}) and (\ref{ltft}),
\begin{displaymath}
\|L_{(y_0,\ldots,y_{2n+1})} - x\|_\infty \leq \left(\frac{N}{2N + 2}\right)^{1/2} \alpha + \epsilon
\end{displaymath}
which, by the arbitrariness of $\epsilon$, entails that $\mu_{\textrm{\upshape{H}},\infty}(\mathcal{F}) \leq \left(\frac{N}{2N + 2}\right)^{1/2} \alpha$, and thus, by the arbitrariness of $\alpha$, the inequality 
$$\mu_{\text{\upshape{H}},\infty}(\mathcal{X}) \leq  \left(\frac{N}{2N + 2}\right)^{1/2} \mu_{\text{\upshape{uec}}}(\mathcal{X})$$
is established.
\end{proof}

\section{Proof of Theorem \ref{thm:qsAA}}\label{sec:qsAA}

We transport the measure of non-uniform tightness from $\mathcal{P}(\mathcal{C})$ to the collection of csp's via their probability distributions. Thus, for a set $\Xi$ of csp's,
$$\mu_{\text{\upshape{ut}}}(\Xi) = \sup_{\epsilon > 0} \myinf_{\mathcal{F}} \sup_{\xi \in \Xi} \mathbb{P}\left(\xi \notin \cup_{y \in \mathcal{F}} B_\infty(y,\epsilon)\right),$$ 
where the infimum is taken over all finite sets $\mathcal{F} \subset \mathcal{C}$, and 
$$B_\infty(y,\epsilon) = \{x \in \mathcal{C} \mid \|y -x\|_\infty < \epsilon\}.$$

Before giving the proof of Theorem \ref{thm:qsAA}, we state three lemmas, which are easily seen to follow from the definitions.

\begin{lem}\label{lem:characMut}
Let $\Xi$ be a collection of csp's and $\alpha \in \mathbb{R}^+_0$. Then the following assertions are equivalent.
\begin{enumerate}
\item $\mu_{\text{\upshape{ut}}}(\Xi) < \alpha.$
\item For each $\epsilon > 0$ there exists a uniformly bounded set $\mathcal{X} \subset \mathcal{C}$ such that 
\begin{enumerate}
\item $\mu_{\text{\upshape{H}},\infty}(\mathcal{X}) < \epsilon,$
\item $\forall \xi \in \Xi : \mathbb{P}(\xi \notin \mathcal{X}) < \alpha.$
\end{enumerate}
\end{enumerate}
\end{lem}

\begin{lem}\label{lem:characMsub}
Let $\Xi$ be a collection of csp's and $\alpha \in \mathbb{R}^+_0$. Then the following assertions are equivalent.
\begin{enumerate}
\item $\mu_{\text{\upshape{sub}}}(\Xi) < \alpha.$
\item There exists a uniformly bounded set $\mathcal{X} \subset \mathcal{C}$ such that 
$$\forall \xi \in \Xi : \mathbb{P}(\xi \notin \mathcal{X}) < \alpha.$$
\end{enumerate}
\end{lem}

\begin{lem}\label{lem:characMsuec}
Let $\Xi$ be a collection of csp's and $\alpha \in \mathbb{R}^+_0$. Then the following assertions are equivalent.
\begin{enumerate}
\item $\mu_{\text{\upshape{suec}}}(\Xi) < \alpha.$
\item For each $\epsilon > 0$ there exists a set $\mathcal{X} \subset \mathcal{C}$ such that 
\begin{enumerate}
\item $\mu_{\text{\upshape{uec}}}(\mathcal{X}) < \epsilon,$
\item $\forall \xi \in \Xi : \mathbb{P}(\xi \notin \mathcal{X}) < \alpha.$
\end{enumerate}
\end{enumerate}
\end{lem}

\begin{proof}[Proof of Theorem \ref{thm:qsAA}]
Let $\Xi$ be a collection of csp's. By Theorem \ref{thm:qProkh}, 
$$\sup_{\lambda \in \mathbb{R}^+_0} \mu_{\text{\upshape{H}},\lambda} (\Xi) = \mu_{\text{\upshape{ut}}}(\Xi),$$
whence it suffices to show that
$$\max\{\mu_{\text{\upshape{sub}}}(\Xi), \mu_{\text{\upshape{suec}}}(\Xi)\} \leq \mu_{\text{\upshape{ut}}}(\Xi) \leq \mu_{\text{\upshape{sub}}}(\Xi) + \mu_{\text{\upshape{suec}}}(\Xi).$$

We first establish that
$$\mu_{\text{\upshape{ut}}}(\Xi) \leq \mu_{\text{\upshape{sub}}}(\Xi) + \mu_{\text{\upshape{suec}}}(\Xi).$$

Fix $\epsilon > 0$, and $\alpha,\beta \in \mathbb{R}^+_0$ such that
$$\mu_{\text{\upshape{sub}}}(\Xi) < \alpha$$
and
$$\mu_{\text{\upshape{suec}}}(\Xi) < \beta.$$
By Lemma \ref{lem:characMsub}, there exists a uniformly bounded set $\mathcal{Y} \subset \mathcal{C}$ such that
$$\forall \xi \in \Xi : \mathbb{P}(\xi \notin \mathcal{Y}) < \alpha,$$
and, by Lemma \ref{lem:characMsuec}, there exists a set $\mathcal{Z} \subset \mathcal{C}$ such that
\begin{equation}
\mu_{\text{\upshape{uec}}}(\mathcal{Z}) < \left(\frac{N}{2N + 2}\right)^{-1/2} \epsilon\label{eq:muecZ}
\end{equation}
and
$$\forall \xi \in \Xi : \mathbb{P}(\xi \notin \mathcal{Z}) < \beta.$$
Put 
$$\mathcal{X} = \mathcal{Y} \cap \mathcal{Z}.$$
Then $\mathcal{X}$ is uniformly bounded. Also, by Theorem \ref{thm:qAA} and (\ref{eq:muecZ}),
$$\mu_{\text{\upshape H},\infty}(\mathcal{X}) \leq \left(\frac{N}{2N + 2}\right)^{1/2} \mu_{\text{\upshape{uec}}}(\mathcal{X}) \leq \left(\frac{N}{2N + 2}\right)^{1/2} \mu_{\text{\upshape{uec}}}(\mathcal{Z}) < \epsilon,$$
and, for $\xi \in \Xi$,
$$\mathbb{P}(\xi \notin \mathcal{X}) \leq \mathbb{P}(\xi \notin \mathcal{Y}) + \mathbb{P}(\xi \notin \mathcal{Z}) < \alpha + \beta.$$
We conclude from Lemma \ref{lem:characMut} that
$$\mu_{\text{\upshape{ut}}}(\Xi) < \alpha + \beta,$$
from which the desired inequality follows.

Next, we prove that 
$$\max\{\mu_{\text{\upshape{sub}}}(\Xi), \mu_{\text{\upshape{suec}}}(\Xi)\} \leq \mu_{\text{\upshape{ut}}}(\Xi).$$

Fix $\epsilon > 0$, and $\alpha \in \mathbb{R}^+_0$ such that 
$$\mu_{\text{\upshape{ut}}}(\Xi) < \alpha.$$
By Lemma \ref{lem:characMut}, there exists a uniformly bounded set $\mathcal{X} \subset \mathcal{C}$ such that
\begin{equation}
\mu_{\text{\upshape{H}},\infty}(\mathcal{X}) < \epsilon/2\label{eq:muH}
\end{equation}
and
$$\forall \xi \in \Xi : \mathbb{P}\left(\xi \notin \mathcal{X}\right) < \alpha.$$ 
We conclude from Lemma \ref{lem:characMsub} that
$$\mu_{\text{\upshape{sub}}}(\Xi) < \alpha.$$
Moreover, by Theorem \ref{thm:qAA} and (\ref{eq:muH}), 
$$\mu_{\text{\upshape{uec}}}(\mathcal{X}) \leq 2 \mu_{\text{\upshape{H}},\infty}(\mathcal{X}) < \epsilon,$$
and Lemma \ref{lem:characMsuec} allows us to infer that
$$\mu_{\text{\upshape{suec}}}(\Xi) < \alpha,$$
which finishes the proof of the desired inequality.
\end{proof}

\section{Conclusions}

In this work, we have quantified Prokhorov's Theorem by establishing an explicit formula for the Hausdorff measure of non-compactness (HMNC) for the parametrized Prokhorov metric on the set of Borel probability measures on a Polish space (Theorem \ref{thm:qProkh}). Furthermore, we have quantified the Arzel{\`a}-Ascoli Theorem by obtaining an upper and a lower estimate for the HMNC for the uniform norm on the space of continuous maps of a compact interval into Euclidean N-space, using Jung's Theorem on the Chebyshev radius (Theorem \ref{thm:qAA}). Finally, we have combined the obtained results to quantify the stochastic Arzel{\`a}-Ascoli Theorem by providing an upper and a lower estimate for the HMNC for the parametrized Prokhorov metric on the set of multivariate continuous stochastic processes (Theorem \ref{thm:qsAA}). This work fits nicely in the research initiated in \cite{BLV11}, the aim of which is to systematically study quantitative measures, such as the HMNC, in the realm of probability theory.

\section*{Competing interests}

The author declares that there is no conflict of interests.

\section*{Authors' contributions}

The author declares to be the only contributor to this work.

\section*{Acknowledgements}

The author thanks Mark Sioen for interesting discussions concerning the topics in this work, and the Fund for Scientific Research Flanders (FWO) for its financial support.

\end{document}